\documentclass{article}
\usepackage{amsmath, amssymb, amsthm}
\usepackage{thesisCommands}
\usepackage{float}
\usepackage{tikz}
\usetikzlibrary{decorations.pathreplacing}
\usepackage{bookmark}

\usepackage[
backend=biber,
style=alphabetic,
sorting=ynt
]{biblatex}
\title{The No-Core Principle for Stationary Actions and Ends of Stationary Random Subgroups}
\author{Yair Hartman, Nadav Kalma}
\date{}

\newtheorem{theorem}{Theorem}[section]
\newtheorem{lemma}[theorem]{Lemma}
\newtheorem{col}[theorem]{Corollary}
\theoremstyle{definition}
\newtheorem{definition}[theorem]{Definition}
\newtheorem{remark}[theorem]{Remark}
\newtheorem{question}[theorem]{Question}

\begin{document}

\maketitle

\begin{abstract}
We prove a no-core principle for stationary actions of countable groups.
Namely, if a Borel set intersects almost every orbit in finitely many points
and has positive measure, then it is supported, modulo null sets, on the
finite-orbit part of the action. This extends to stationary actions a basic
regularity phenomenon known for measure-preserving actions.

We apply this principle to the geometry of Stationary Random Subgroups.
For a finitely generated group, we prove that the Schreier graph of a
stationary random subgroup has almost surely \(0,1,2\), or infinitely many
ends. Finally, we contrast this probabilistic regularity with the topological
notion of Boomerang subgroups: for every \(k\geq 3\), including
\(k=\aleph_0\), we construct a Boomerang subgroup of \(\mathbb{F}_3\) whose Schreier
graph has exactly \(k\) ends.
\end{abstract}

\section{Introduction}

In the theory of group actions on measure spaces, one considers probability measure-preserving (p.m.p.)~actions and non-singular actions, namely those that preserve the measure class. Non-singular actions provide a broad framework requiring minimal rigid structure. While every p.m.p.\ action is inherently non-singular, the former exhibits a wide variety of strong regularity properties.

Of particular interest to our work is the fact that, on the infinite-orbit part, p.m.p.\ actions do not admit Borel selectors of positive measure. 

Stationary actions occupy an intermediate realm; these are actions equipped with a measure that is invariant under convolution with respect to a probability measure $\mu\in \ProbG$ (a ``random walk'') on the group. Every p.m.p.\ action is $\mu$-stationary, for every $\mu$, and every $\mu$-stationary action is non-singular, assuming that $\mu$ is \emph{generating}, namely, the support of $\mu$ generates $\G$ as a semigroup. A key advantage of stationary measures over p.m.p.\ actions is their guaranteed existence on compact spaces, even for actions of non-amenable groups. At the same time, stationary measures are more ``regulated'' than general non-singular actions.

A natural question arises: which of the properties and regularities of p.m.p.\ actions extend to the broader context of stationary actions?

Our first result shows that if a stationary action admits a positive-measure ``core'', then that core is necessarily supported on finite orbits.

\begin{theorem}[No-core principle for stationary actions]
    \label{no-core}
    Let $\mu \in \ProbG$ be a generating measure, and let $(\Gamma, \mu) \curvearrowright (X,\nu)$ be a stationary action. Let $C \subseteq X$ be a measurable set with $\nu(C)>0, $ such that for $\nu$-a.e. $x\in X$, $|\G x\cap C|< \infty$. Then for $\nu$-a.e.\ $x \in C$, $|\Gamma x| < \infty$. 
    
    Equivalently, almost every point in $C$ has a finite orbit.
\end{theorem}
\noindent We call a set $C$ as in the theorem above a \emph{core}.

\vspace{1mm}

Our result holds in particular for p.m.p.\ actions, as they are $\mu$-stationary for every generating $\mu$. This result for p.m.p.\ actions is known using tools of equivalence relations and Descriptive Set Theory (see, for example~\cite{gaboriau2025around}). It is interesting that our proof uses random walks to prove this result, even in the p.m.p.\ setup. One should note that the no-core principle does not hold for general non-singular actions.

The No-Core Principle was first introduced in \cite{Ab_rt_2022} to prove a result regarding the space of ends for unimodular random manifolds. We apply this principle to address a similar question in the theory of Stationary Random Subgroups (SRSs).

Let $\G$ be a discrete countable group, 
let $\subg$ denote the compact space of all subgroups of $\G$, equipped with the Chabauty topology, and with the conjugation action of $\G$. A random subgroup is a 
probability measure on $\subg$.
An Invariant Random Subgroup (IRS), introduced in \cite{kes14}, is a conjugation-invariant probability measure on \(\subg\). Equivalently, the conjugation action \(\Gamma\curvearrowright(\subg,\nu)\) is p.m.p. 
Similarly, a $\mu$-Stationary Random Subgroup (SRS) is a $\mu$-stationary measure on $\subg$.

By viewing a subgroup geometrically via its Schreier graph, one can ask which geometric properties of normal subgroups generalize to IRSs, and subsequently, which properties of IRSs extend to SRSs. In this article, we address this question concerning the ends of a graph.

For a locally finite graph, the space of ends is a topological invariant that formalizes the notion of ``directions to infinity''. The classical theorem of Freudenthal \cite{Freudenthal1944/45} and Hopf \cite{Hopf1943/44} establishes that the Cayley graph of a finitely generated group has exactly $0, 1, 2$, or infinitely (uncountable) many ends. Furthermore, this number is independent of the choice of generating set, making it a group invariant. Building on this, Stallings~\cite{Stallings1971GroupTA} proved that the number of ends reveals deep algebraic information about the group; for example, a group is 2-ended if and only if it is virtually $\mathbb{Z}$.

Motivated by these classical results, we say a collection of graphs is \emph{end-rigid} if each of its graphs has exactly $0, 1, 2$, or infinitely many ends. Thus the collection of Cayley graphs of finitely generated groups is end-rigid. Equivalently, the Schreier graphs associated with normal subgroups (which are Cayley graphs of quotients), are end-rigid.
Abels~\cite{Abels1973} proved that the collection of quasi-transitive graphs (i.e., all graphs with a finite number of orbits under their automorphism group) is end-rigid. In terms of $\subg$, all Schreier graphs of subgroups with finite conjugacy classes are end-rigid.

Consequently, given a probability measure on $\subg$, one may ask whether the Schreier graph of almost every subgroup $\Delta < \G$ is end-rigid. For IRSs, a theorem of Lyons and Schramm \cite{LS99} regarding unimodular graphs provides an affirmative answer. Moreover, Curien~\cite[Corollary 24]{curien18} proved that the stationary random graphs, with respect to nearest neighbor transitions, are end-rigid as well.

In this article, we prove that this rigidity phenomenon also holds for general SRSs.

\begin{theorem}
\label{ends_theorem}
Let $\G$ be a group generated by a finite symmetric set $S$, and let $\mu \in \ProbG$ be a generating measure. Let \(\nu\in\Prob(\subg)\) be a \(\mu\)-stationary random subgroup.
Then
\[
|\Ends(\Sch(H\backslash\Gamma,S))|\in\{0,1,2,\infty\}
\]
for \(\nu\)-almost every \(H\).
\end{theorem}

Here, $\Ends(\Sch(H\backslash\Gamma,S))$ denotes the ends space of the Schreier graph of subgroup $H$ with respect to the generating set $S$.

For IRSs, unimodularity gives the stronger conclusion that the infinite case is uncountable. Our argument for SRSs rules out finite numbers of ends larger than two, but does not rule out countably infinitely many ends.

Simultaneously with and independently of our work, Levit and Silman~\cite{Levit26} proved a related result in a slightly different setting, using different tools. Their argument applies to the general framework of random metric-measure spaces and they also show that when the end space is infinite, it is topologically, then it is a Cantor set. Moreover, they classify
the homeomorphism types of surfaces, generalizing the work of Biringer-Raimbault~\cite{Biringer_2016} on unimodular Riemannian manifolds
to the stationay setup.

In the finitely generated group setting, however, the approach taken in~\cite{Levit26} requires the random walk to be finitely supported and symmetric, assumptions that are not needed in our argument.
This leaves open the following natural question: can a \(\mu\)-SRS for an infinitely supported or non-symmetric measure have countably infinitely many ends almost surely? \\

While SRSs provide a probabilistic generalization of IRSs, one can also consider the topological related notion of  {Boomerang subgroups}. Introduced in \cite{Boomerang}, a subgroup $\Delta \le \G$ is called a Boomerang subgroup if for every $g \in \G$, the sequence of conjugates $(g^n \Delta g^{-n})_{n=0}^\infty$ admits a subsequence converging to $\Delta$ in the Chabauty topology. The relation to IRSs is the following:  if $\lambda$ is any IRS on $\G$, then $\lambda$-almost every subgroup is a Boomerang subgroup; that is, $\lambda(\mathrm{Boom}(\G)) = 1$, where $\mathrm{Boom}(\G)$ denotes the space of all Boomerang subgroups of $\G$.

In contrast to the almost sure end-rigidity established for IRSs and SRSs, we demonstrate that end-rigidity does not hold for general Boomerang subgroups.

\begin{theorem}[Failure of end-rigidity for Boomerang subgroups]
\label{thm:boomerang_ends}
    For any integer $k \ge 3$ or $k = \aleph_0$, there exists a Boomerang subgroup of the free group $\mathbb{F}_3$ whose Schreier graph has exactly $k$ ends.    
\end{theorem}

\paragraph{Acknowledgments.}
We would like to thank Yair Glasner for his insightful comments and for his assistance in constructing the example for Theorem~\ref{thm:boomerang_ends}, and Ariel Yadin for insteresting and useful discussions. 

This work was supported by the funding from the European Research Council
(ERC) under the European Union’s Horizon Europe research and innovation programme, Grant agreement No. 101078193.

\section{Preliminaries}

\subsection*{Non-singular actions, invariant measures, and stationary measures}

Let $\Gamma$ be a discrete countable group acting on a compact topological space $X$ by measurable automorphisms, making $X$ a $\Gamma$-space. This action naturally extends to the space of Borel probability measures $\Prob(X)$ via the formula $(g.\nu)(E) = \nu(g^{-1}E)$. 

A measure $\nu \in \Prob(X)$ is \emph{ergodic} if every $\Gamma$-invariant Borel set has measure $0$ or $1$. The action of $\Gamma$ on $(X, \nu)$ is \emph{non-singular} if $g.\nu$ and $\nu$ are mutually absolutely continuous (meaning they share the same null sets) for all $g \in \Gamma$. 

A measure is \emph{invariant} if $g.\nu = \nu$ for all $g \in \Gamma$. However, invariant measures on compact spaces are only guaranteed to exist when $\Gamma$ is an amenable group.

A meausre $\mu\in \ProbG$ is called \emph{generating} if the support of the measure generates $\G$ as a semigroup, meaning
$
\bigcup_{n\ge 1}\mathrm{supp}(\mu^{*n})=\Gamma$.  
Given a generating probability measure $\mu \in \Prob(\Gamma)$, a measure $\nu \in \Prob(X)$ is \emph{$\mu$-stationary} if it satisfies the convolution equation $\nu = \mu \ast \nu = \sum_{g\in \Gamma} \mu(g) g.\nu$. 

Stationary measures are highly useful because a $\mu$-stationary measure is guaranteed to exist on any compact $\Gamma$-space, even when invariant measures do not. Furthermore, if $\mu$ is a generating measure, 
any $\mu$-stationary measure is automatically non-singular. This establishes a natural hierarchy: every invariant measure is $\mu$-stationary, and every $\mu$-stationary measure is non-singular.

\subsection*{Markov chains}
We will rely on standard elements of Markov chain theory on infinite state spaces. 

Let $S$ be a countable state space and $P: S \times S \to [0,1]$ be a transition probability matrix, satisfying $\sum_{y\in S}P(x,y) = 1$ for all $x \in S$. We denote by $Z_n$ the position of the chain at time $n$.
For every $x \in S$, let $\mathbb{P}_x$ denote the Markovian probability measure on the path space $S^{\mathbb{N}}$ conditioned on the initial state $Z_0 = x$, with corresponding expectation $\mathbb{E}_x$.

We define the number of visits to a state $y$ as the random variable $V(y) = \sum_{n=0}^\infty \mathbf{1}_{\{Z_n=y\}}$. Its expectation, $\mathbb{E}_x[V(y)]$, is the expected number of visits to $y$ given $Z_0 = x$. 
Similarly, we define the first hitting time of $y$ as $\tau_y^+ = \inf\{t \ge 1 \mid Z_t = y\}$, making $\mathbb{E}_x[\tau_y^+]$ the expected first return time to $y$ starting from $x$. 
A chain is called \emph{irreducible} if for all $x,y \in S$, there exists an integer $t \ge 1$ such that $P^t(x,y) > 0$.

Recall the classical classification of an irreducible Markov chain with transition matrix $P$. The chain is called:
\begin{itemize}
    \item \textbf{Transient} if for some (equivalently all) $x,y \in S$, $\mathbb{E}_x[V(y)] < \infty$.
    \item \textbf{Recurrent} if for some (equivalently all) $x,y \in S$, $\mathbb{E}_x[V(y)] = \infty$. Recurrent chains are further classified as:
    \begin{itemize}
        \item \emph{Positive recurrent} if for some (equivalently all) $x \in S$, $\mathbb{E}_x[\tau_x^+] < \infty$.
        \item \emph{Null recurrent} if for some (equivalently all) $x \in S$, $\mathbb{E}_x[\tau_x^+] = \infty$.
    \end{itemize}
\end{itemize}

The following classical lemma is a key ingredient in our proof of the No-Core Principle.

\begin{lemma}
    \label{prob_vanishes0}
    Let $P$ be an irreducible Markov chain on an infinite state space $S$. If $P$ is \emph{null recurrent} or \emph{transient}, then for any $x,y \in S$,
    \[
    P^n(x,y) \xrightarrow[n \to \infty]{} 0.
    \]
\end{lemma}

\begin{proof}
    If $P$ is transient, the expected number of visits to $y$ is finite. Thus,
    \[
    \sum_{n=0}^\infty P^n(x,y) = \mathbb{E}_x\left[\sum_{n=0}^\infty \mathbf{1}_{\{Z_n=y\}}\right] = \mathbb{E}_x[V(y)] < \infty.
    \]
    Since the infinite series converges, the general term must tend to zero, yielding $P^n(x,y) \xrightarrow[n\to\infty]{} 0$.
    For the case where the chain is null recurrent, we refer to \cite[Theorem 21.19]{LevinPeresWilmer2006}.
\end{proof}

\begin{col}
    \label{prob_vanishes}
    Under the assumptions of Lemma \ref{prob_vanishes0}, for any finite subset of states $F \subseteq S$ and any initial state $x \in S$, 
    \[
    P^n(x,F) := \sum_{y\in F} P^n(x,y) \xrightarrow[n\to \infty]{} 0.
    \]
\end{col}

\begin{proof}
    Since $F$ is a finite set, the limit of the finite sum is the sum of the limits. Applying Lemma \ref{prob_vanishes0} directly yields the result.
\end{proof}

We apply this general theory to the setting of group actions. Given a $\Gamma$-action on a space $X$, a probability measure $\mu \in \Prob(\Gamma)$ induces a Markov chain on each orbit $\mathcal{O}_x = \Gamma x$. The transition probabilities are given by
\[
P_{\mathcal O_x}(x,y) = \sum_{\substack{g\in\Gamma\\ g.x=y}} \mu(g).
\]
Note, the original random walk on $\G$ is a \textit{right} random walk (the position at time $n$ is $g_1 g_2 \cdots g_n$), while the induced Markov chain on $X$ operates via \textit{left} multiplication (the position at time $n$ is $g_n g_{n-1}\cdots g_1.x$).

\subsection*{$\subg$, IRS and SRS}

We denote by $\subg\subseteq 2^\Gamma$ the space of subgroups of $\Gamma$, equipped with the subspace topology (commonly referred to as the Chabauty topology). Note that $\subg$ is a compact metrizable space, and $\Gamma$ acts on it by conjugation: $g. H=gHg^{-1}$.

An \emph{Invariant Random Subgroup} (IRS) is a Borel probability measure $\lambda \in \Prob(\subg)$ which is invariant under the conjugation action of $\Gamma$. 
An IRS can be seen as a probabilistic generalization of a normal subgroup, in the sense that the Dirac measure $\delta_N$ is an IRS if and only if $N \unlhd \Gamma$ is normal. 

Given a probability measure $\mu \in \Prob(\Gamma)$, a \emph{$\mu$-Stationary Random Subgroup} (SRS) is a Borel probability measure $\nu \in \Prob(\subg)$ which is stationary with respect to $\mu$ and the conjugation action. 
Any IRS is an example of an SRS, but there are many more SRSs than IRSs.

\subsection*{Cayley and Schreier graphs}

Let $S$ be a finite symmetric generating set of $\Gamma$. We denote by $\Cay(\Gamma,S)$ the associated Cayley graph, whose vertex set is $\Gamma$ and whose edges are given by
\[
g \xrightarrow{s} gs \qquad (s\in S).
\]

For a subgroup $H\le\Gamma$, we denote by $\Sch(H\backslash\Gamma,S)$ the associated Schreier graph, whose vertex set is $H\backslash\Gamma$ and whose edges are given by
\[
Hg \xrightarrow{s} Hgs \qquad (s\in S).
\]
Both $\Cay(\G,S)$ and $\Sch(H\backslash\Gamma,S)$ are connected, locally finite graphs.

Fixing $S$, we denote $\Sch(H):=\Sch(H\backslash\Gamma,S)$ and by $$\Sch(\Gamma):= \{(\Sch(H), Hx)~:~ H\in \subg, x\in \G \}$$ we denote the space of all rooted Schreier graphs of $\Gamma$. 
This is also a compact metric space when equipped with the local topology (often called the pointwise or Gromov-Hausdorff topology), where two rooted Schreier graphs are at a distance of at most $1/n$ if their metric balls of radius $n$ around the roots are rooted-isomorphic. This ``geometric'' topology coincides with the topology on $\subg$.

The space $\Sch(\Gamma)$ is also a $\Gamma$-space, where the action of $g\in \Gamma$ corresponds to changing the root:
\[
g \cdot (\Sch(H\backslash\Gamma, S), Hx) = (\Sch(H\backslash\Gamma, S), Hxg^{-1}) \quad \text{for all } g\in \Gamma.
\]

Observe that the rooted Schreier graph of a conjugate subgroup is canonically isomorphic to the original Schreier graph with a shifted root. 
\begin{align*}
    &(\Sch(H),Hx) \cong (\Sch(x^{-1}Hx), x^{-1}Hx) \quad &\forall H\in \subg, x\in \G \\
    &Hg \mapsto x^{-1}Hg \quad
    &\forall Hg\in H\backslash \G
\end{align*}

The discussion above summaries as follows:
    The map $\subg \to \Sch(\Gamma)$ given by
    $H \mapsto (\Sch(H\backslash\Gamma, S), H)$
    is a $\Gamma$-equivariant homeomorphism.

Moreover, for every $H\in \subg$ the relation between conjugation of $H$ and isomorphic Schrier graphs is given by:
\begin{lemma}
    \label{lemma:orbits_conj_class}
    Let $H \in \subg$ and let $\Aut(\Sch(H))$ denote the group of label-preserving graph automorphisms of $\Sch(H)$. 
    Then there is a natural bijection between the conjugacy class $H^\Gamma = \{x^{-1}Hx : x \in \Gamma\}$ and the orbits of the vertices of $\Sch(H)$ under the action of $\Aut(\Sch(H))$.
\end{lemma}

Indeed, the vertices of $\Sch(H)$ are the right cosets $Hx$ for $x \in \Gamma$. Two vertices $Hx$ and $Hy$ belong to the same $\Aut(\Sch(H))$-orbit if and only if the rooted Schreier graphs $(\Sch(H), Hx)$ and $(\Sch(H), Hy)$ are isomorphic. 
    Because changing the root of the Schreier graph corresponds to conjugating the subgroup, the subgroup corresponding to the Schrier graph with root at $Hx$ is exactly $x^{-1}Hx$, this isomorphism holds if and only if $x^{-1}Hx = y^{-1}Hy$. 
    Thus, the orbits of $\Aut(\Sch(H))$ uniquely correspond to the distinct conjugates of $H$.

\subsection*{Ends of graphs and random subgroups}
\label{Ends_Chapter}

The space of ends of a locally finite graph $G$, denoted by $\Ends(G)$, formalizes the ``ways to go to infinity'' in the graph. It is constructed as the inverse limit of the connected components of the complement of finite balls.
This construction endows $\Ends(G)$ with  a compact topological structure (see for example,~\cite{kron2005introduction}).

The classical Freudenthal-Hopf \cite{Freudenthal1944/45,Hopf1943/44} theorem dictates that the number of ends of a finitely generated Cayley graph is always $0, 1, 2$, or $2^{\aleph_0}$. Abels~\cite{Abels1973} later generalized this result, proving that the same holds for any quasi-transitive graph.

One gets a  measurable function $\Ends: \subg \to \mathbb{N} \cup \{0, \infty\}$ assigning $H$ the cardinality of $\Ends(\Sch(H\backslash\Gamma, S))$. Because conjugating $H$ merely re-roots the graph without altering its global end structure, this function is $\Gamma$-invariant. Consequently, for any ergodic measure on $\subg$, the number of ends is almost surely constant.

For normal subgroups (whose Schreier graphs are Cayley graphs) and subgroups with finitely many conjugates (which yield quasi-transitive Schreier graphs), the Freudenthal-Hopf and Abels theorems directly guarantee that the number of ends is $0, 1, 2$, or $2^{\aleph_0}$. 

Lyons and Schramm \cite{LS99} extended this result by showing that any unimodular random graph almost surely has $0, 1, 2$, or $2^{\aleph_0}$ ends. Since the Schreier graph of an Invariant Random Subgroup (IRS) inherently corresponds to a unimodular random graph \cite{Biringer_2016}, it follows that any IRS almost surely satisfies this same restriction on its ends. For an analogous generalization to unimodular random manifolds, see \cite{end_manifolds,Ab_rt_2022}.

\section{The no core principle for stationary actions}

\subsection{Markov chains on orbits}
Let $\G$ be a discrete countable group. Let $\mu \in \Prob(\G)$ be a generating measure 
and let $X$ be a measurable $\G$-space. For every orbit  $O_x:=\Gamma.x$, there is a naturally defined  \emph{Markov Chain}, induced by the random walk on the group itself. The transition probabilities are given by, 
$$P_{O_x}(y,z) = \sum_{\substack{g \in \Gamma \\
g.y = z}}\mu(g)$$
for any $y,z \in \OX$.

Note that, as usual, we consider a right random walk on $\G$, the resulting Markov Chain on an orbit becomes a left one. 
One can see that 
\begin{align*}
P_{\OX}^n(y,z)
&= \sum_{\substack{g \in \Gamma \\
g.y=z}}\mu^{*n}(g).
\end{align*}

Since we assumed that $\mu$ is generating, the Markov chain is irreducible.

\begin{lemma}\label{lemma:dic-for-orbits}
    Fix some $x\in X$, and consider the Markov Chain $P_{\OX}$. Then, either
    \begin{enumerate}
        \item $|\OX| < \infty$.
        \item for every finite set \(F\subseteq \mathcal O_x\) and every \(y\in\mathcal O_x\),
    $$P_{\OX}^n(y,F) = \sum_{z\in F }\sum_{\substack{g \in \Gamma \\
g.y=z}}\mu^{*n}(g) \xrightarrow[n\to \infty]{} 0$$

    \end{enumerate}
\end{lemma}
\begin{proof}

    We show that $|\OX| < \infty$ if and only if $P_{\OX}$ is positive recurrent.

    Note that this would finish the proof, as the other case, that the orbit is infinite, implies that $P_{\OX}$ is either null recurrent or transient. In both cases, we have that (2) as we showed in  Corollary~\ref{prob_vanishes}.

    It is easy to see that a finite irreducible Markov chain is positive recurrent.

    Assume conversely that \(P_{\mathcal O_x}\) is positive recurrent. Let
$K=\Stab_\Gamma(x)$, and let \(Z_n=g_n\cdots g_1\) be the left \(\mu\)-random walk on \(\Gamma\).
Then the orbit chain started at \(x\) is given by \(Z_nx\). Therefore the
first return time of the orbit chain to \(x\) is exactly
\[
\tau_K^+=\inf\{n>0:Z_n\in K\}.
\]
Thus positive recurrence of \(P_{\mathcal O_x}\) implies that \(K\) has
finite expected return time for the left \(\mu\)-random walk.

Passing to inverses gives the usual right \(\check\mu\)-random walk, where
\[
\check\mu(g)=\mu(g^{-1}).
\]
By Kac's formula~\cite[Theorem~1.2]{HARTMAN_2013},
\[
[\Gamma:K]=\mathbb E_{\check\mu}[\tau_K^+]<\infty.
\]
Hence
\[
|\mathcal O_x|=[\Gamma:\Stab_\Gamma(x)]<\infty.
\]
\end{proof}

\subsection{Core and stationary actions}

In \cite[Theorem 1.15]{Ab_rt_2022}, the authors introduce the notion of a \emph{core} for a random manifold and formulate the ``No-Core Principle'' which characterizes the conditions under which a core may or may not exist. \\
Drawing inspiration from their work, we generalize the definition of a core to arbitrary $\G$-spaces as follows.

\begin{definition}[Core]
    Let $(X,\nu)$ be a probability $\G$-space. A Borel subset $C \subseteq X$ is called a Core of the action if for $\nu$ a.e.\ $x \in X$, 
    $$|\Gamma.x \cap C | < \infty$$
    In words, a Core is a way to select from each orbit a finite set, in a Borel fashion.
\end{definition}
For example, if every orbit is finite, then every measurable set is a core.

\noindent More interestingly, consider a Borel selector. Namely, let $(X,m)$ be a probability measure preserving action of $\mathbb{Z}$, and let $C$ be a Borel set that intersects every $\mathbb{Z}$ exactly at one point. It is easy to see that by (the classical) Kac's formula, in that case, $m$-almost every orbit must be finite. 

We think of this setup, when the orbits are finite, as trivial. To get a non-trivial example of an action with a core, one can consider a transitive $\G$-action (for some $\G$) on an infinite (necessarily countable) set $X$, equipped with a fully supported measure $\eta\in \Prob(X)$. In that case, every finite subset is a core, although the orbit is infinite.

Our main theorem states that for stationary actions, the only case in which there is a (non-trivial) core, it is due to a trivial reason: finite orbits.

\begin{theorem}[No-core principle for stationary actions]
    \label{no-core2}
    Let $\mu \in \ProbG$ be a generating measure, and let $(\Gamma, \mu) \curvearrowright (X,\nu)$ be a stationary action. Let $C \subseteq X$ be a core with positive measure. Then for $\nu$-a.e.\ $x \in C$, $|\Gamma x| < \infty$. 
\end{theorem}
In words, almost every point in a core has a finite orbit.

\begin{remark}
Recall that every p.m.p. action is $\mu$-stationary (for any generating $\mu\in \ProbG$), so this proof implies a no-core principle for p.m.p. actions. For p.m.p. action of countable groups, one can prove it using tools from Descriptive Set Theory (see, for example, Gaboriau's~\cite[Section 1.3]{gaboriau2025around}). It is interesting to note that our proof uses random walks to prove this result regarding pmp actions. Of course, the setup of $\mu$-stationary is strictly larger than pmp actions, but unlike pmp or non-singular actions, it is not clear how to model it in terms of Descriptive Set Theory. 
\end{remark}

\begin{proof}
    Let $C\subseteq X$ be a core, let $D:= \{x \in C~ \vert ~ |\G.x| = \infty \}$ be the set of all points of infinite orbit. Since $\Gamma$ is countable and the action is Borel, the set of points with finite orbit is Borel. It follows that $D$ is Borel.
    We claim that $\nu(D) =0$.
    Assume towards contradiction that $\nu(D) >0$, then $0 < \nu(D) = \int_{X} \mathbbm{1}_Dd\nu(x)$.\\
    Now since $\nu$ is stationary, for any $n \ge 1$,
    \begin{align*}
    0 <\int_{X} \mathbbm{1}_D(x)d\nu(x) 
    &= \sum_{g\in \Gamma}\mu^{*n}(g)\int_{X} \mathbbm{1}_D(x)d(g.\nu)(x) \\
    &= \sum_{g\in \Gamma}\mu^{*n}(g)\int_{X} \mathbbm{1}_D(g.x)d\nu(x) \\
    &= \int_{X} \sum_{g\in \Gamma}\mu^{*n}(g)\mathbbm{1}_D(g.x)d\nu(x).
    \end{align*} 
Where the last equality is due to Fubini–Tonelli Theorem. Now consider the functions 
$$f_n(x) :=\sum_{g\in \Gamma}\mu^{*n}(g)\mathbbm{1}_D(g.x)$$
Then
\begin{align}
    \label{equation}
    \nu(D) = \int_{X}f_n(x)d\nu(x) \quad  \text{for every } n \ge 1 
    \tag{\(*\)}
\end{align}

We will show that \(f_n(x)\to 0\) for \(\nu\)-a.e. \(x\). Let $X_0=\{x\in X:|\Gamma x\cap C|<\infty\}$.
By the definition of a core, \(\nu(X_0)=1\). Fix \(x\in X_0\).
If \(\Gamma x\cap D=\emptyset\), then \(\mathbf 1_D(gx)=0\) for every \(g\in\Gamma\), and hence \(f_n(x)=0\) for all \(n\).

Assume instead that \(\Gamma x\cap D\neq\emptyset\). Then \(\Gamma x\) is infinite. Since every point in an infinite orbit has the same infinite orbit, we have
\[
D\cap \Gamma x = C\cap \Gamma x.
\]

 For every $x\in X_0$, denote $F_x :=  \Gamma.x \cap C \subseteq \OX$, which is a finite set since $C$ is assumed to be a core.
We now deduce that, 
\begin{align*}
    f_n(x) =& \sum_{g\in \Gamma}\mu^{*n}(g)\ind{D}(g.x)
    \\
    =& \sum_{g\in \Gamma}\mu^{*n}(g)\ind{\{g.x \in F_x\}}\\
    =& \sum_{y\in F_x}\sum_{\substack{g\in \Gamma \\ g.x = y
    }}\mu^{*n}(g) = P_{\OX}^n(x,F_x) \end{align*}

Since the orbit is infinite, by Lemma~\ref{lemma:dic-for-orbits},$ P_{\OX}^n(x,F_x)\to 0$ as $n \to \infty$. Therefore $f_n(x) \to  0$ as $n \to \infty$, for every $x\in X_0$, and recall that $\nu(X_0)=1$.
Since $f_n$ is bounded for any $n$ (by the value 1), by Lebesgue Dominated Convergence Theorem, 
$$\int_{X}f_n(x)d\nu(x) \xrightarrow[n \to \infty]{} 0$$

By equation (\ref{equation}), for every $n$,
$\nu(D) = \int_{X}f_n(x)d\nu(x)$,
so we can conclude that $\nu(D) = 0$, hence a contradiction. 
\end{proof}

\begin{col}
If \(\nu\) is ergodic and there exists a core \(C\) with positive measure, then for \(\nu\)-a.e. \(x\in X\), \(|\Gamma x|<\infty\).
\end{col}

\begin{proof}
Let
\[
X_{\mathrm{fin}}=\{x\in X:|\Gamma x|<\infty\}.
\]
This is an invariant measurable subset of \(X\). By Theorem~\ref{no-core2},
\[
\nu(C\cap X_{\mathrm{fin}})=\nu(C)>0.
\]
Hence \(\nu(X_{\mathrm{fin}})>0\). Since \(\nu\) is ergodic and \(X_{\mathrm{fin}}\) is invariant, \(\nu(X_{\mathrm{fin}})=1\).
\end{proof}

\begin{question}
    Can the definition of core and the no core principle be generalized to locally compact second-countable groups?
\end{question}

\section{Ends of stationary random subgroups}

Recall the discussion above
regarding the ends of a graph of Cayley graphs and Schreier graphs. 
In this section we generalize the end-rigidity results for Cayley graphs and for IRSs, by proving that for a $\mu-$Stationary Random Subgroup the number of ends of the associated Schreier graph is almost surely \(0,1,2\), or infinite.
The theorem will use the ``No-Core Principle'' we proved above, by constructing a core for a subgroup with finitely many ends, greater than $2$.

Throughout this section, \(\Gamma\) and its finite symmetric generating set \(S\) are fixed. We write
\[
\Sch(H):=\Sch(H\backslash\Gamma,S),
\qquad
\Ends(H):=\Ends(\Sch(H)).
\]
We also use the convention \(H^g=g^{-1}Hg\).

\begin{lemma}
\label{finite_ends_infinite_orbit}
    Let $H \in \subg$ be with $2<|\Ends(H)|<\infty$,
    then the conjugacy class $H^\G= \{H^g : g \in  \G \}$ is infinite. 
\end{lemma}
\begin{proof}
    Since $2<|\Ends(H)|<\infty$ we know by Abels' Theorem~\cite{Abels1973} that the action of $\Aut(\Sch(H))$ on $\Sch(H)$ has infinitely many orbits.
    By Lemma~\ref{lemma:orbits_conj_class},
    there is a correspondence between two vertices in the same orbit of $\Aut(\Sch(H))$ and conjugation classes of $H$,
    and hence $H^\G$ is infinite.
\end{proof}

\begin{lemma}
    \label{set_is_a_core}
    Let $\mu\in \ProbG$ generating measure and let $\nu$ be an $\mu$-SRS of $\G$.
    
    For $2<k< \infty$ and $n>1$,
    let $E_k\subseteq \subg$ be a measurable set such that
\[
E_k \subseteq\{H\in\subg:|\Ends(H)|=k\},
\]
and let
\[
C_{k,n}=\left\{H\in E_k:\Sch(H)\setminus B_n(H)\text{ has exactly }k\text{ infinite connected components}\right\}.
\]
Then \(C_{k,n}\) is a Borel core for the conjugation action \(\Gamma\curvearrowright\subg\).
\end{lemma}
\begin{proof}
    Fix $2<k<\infty$ and $n\in \mathbb{N}$, and denote $C:=C_{k,n}$.
    The Borelness of \(C\) is standard in the Chabauty topology: it can be expressed in terms of finite Schreier balls together with countable conditions asserting the existence, or nonexistence, of pairwise disjoint rays outside those balls.
    
    We prove the core property, for $H \in C$, 
    we claim that $$H^{\G} \cap C \subseteq\{g^{-1}Hg ~ : ~ d(H,Hg)_{Sch(H)} \le 2n \}$$ 
    Assume not, then there is some $g_0\in \G$ such that
    $d(H,Hg_0)_{\Sch(H)} >2n$ and $g_0^{-1}Hg_0 \in C$, thus after removing the ball of radius $n$ around $g_0^{-1}Hg_0$ from $\Sch(g_0^{-1}Hg_0)$ we get $k$ connected components.
    Since the action of conjugation corresponds to re-rooting, $B_n(g_0^{-1}Hg_0) \subseteq \Sch(g_0^{-1}Hg_0)$ is isomorphic to $B_n(Hg_0) \subseteq \Sch(H)$, thus after removing the ball of radius $n$ around $Hg_0$ in $\Sch(H)$ we get $k$ connected components. 

    In this point we have two vertices in $\Sch(H)$, $H$ and $Hg_0$, for which the graph minus the ball of radius $n$ around each of the vertices has $k$ infinite connected components.

    But their distance from each other on the graph is more than $2n$, then $B_n(H) \cap B_n(Hg_0) = \emptyset$.
    Thus after removing both of the balls we see that $\Sch(H)$ has at least $2k-2$ infinite connected components,
    and hence $\Sch(H)$ has at least $2k-2$ ends. \\
    But since $2<k<\infty$ then $2k-2>k$ and we get a contradiction to the assumption that $H\in E_k$ meaning $H$ has $k$ ends.\\

    Thus every \(g_0\) with \(g_0^{-1}Hg_0\in C\) satisfies
\[
d_{\Sch(H)}(H,Hg_0)\le 2n.
\]
The ball \(B_{2n}(H)\) is finite, since \(S\) is finite. Moreover, if \(Hg=Hg'\), then \(g^{-1}Hg=(g')^{-1}Hg'\). Therefore the set
\[
\{g^{-1}Hg:d_{\Sch(H)}(H,Hg)\le 2n\}
\]
is finite, and so \(H^\Gamma\cap C\) is finite for every \(H\in C_{k,n}\).
Thus \(C\) is a core.
    
\end{proof}

\begin{figure}[htpb]
    \centering
    \begin{tikzpicture}[scale=1, thick]
        \coordinate (H) at (0,0);
        \coordinate (Hg) at (6,0);

        \draw (H) circle (1.5cm);
        \draw (Hg) circle (1.5cm);

        \filldraw (H) circle (2pt) node[below=2pt] {$H$};
        \filldraw (Hg) circle (2pt) node[below=2pt] {$Hg_0$};

        \node at (0, 1.8) {$B_n(H)$};
        \node at (6, 1.8) {$B_n(Hg_0)$};

        \draw (1.5, 0.3) -- (4.5, 0.3);
        \draw (1.5, -0.3) -- (4.5, -0.3);

        \foreach \angle in {120, 150, 180, 210, 240} {
            \draw (H) ++(\angle:1.5cm) -- ++(\angle:1.5cm);
        }

        \foreach \angle in {60, 30, 0, -30, -60} {
            \draw (Hg) ++(\angle:1.5cm) -- ++(\angle:1.5cm);
        }

        \draw[decorate, decoration={brace, amplitude=10pt, mirror}] 
            (-3.0, -3.0) -- (-0.5, -3.0) node[midway, below=12pt] {$k-1$ ends};
            
        \draw[decorate, decoration={brace, amplitude=10pt, mirror}] 
            (6.5, -3.0) -- (9.0, -3.0) node[midway, below=12pt] {$k-1$ ends};

    \end{tikzpicture}
    \caption{If $d(H, Hg_0) > 2n$, the balls $B_n(H)$ and $B_n(Hg_0)$ are disjoint. Removing both leaves at least $(k-1) + (k-1) = 2k-2$ infinite components, contradicting the assumption that the graph has exactly $k$ ends.}
    \label{fig:2k_minus_2_ends}
\end{figure}
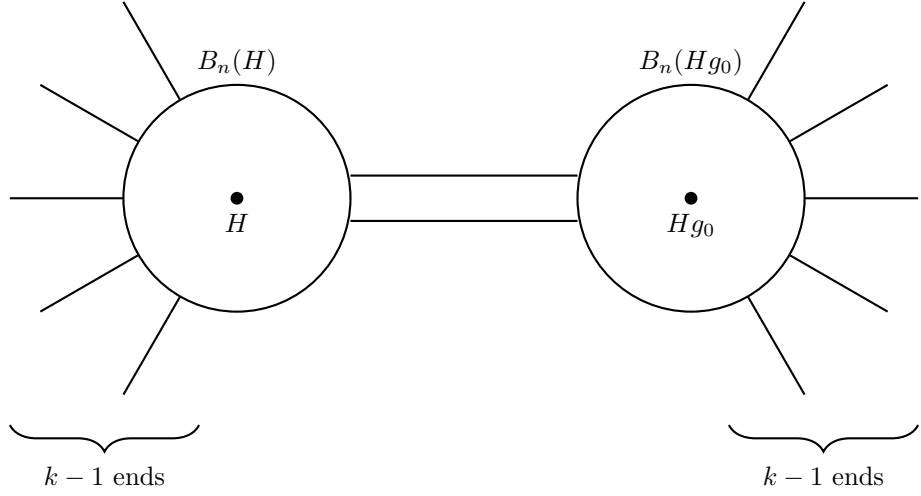

\begin{theorem}
\label{ends_theorem2}
Let \(\Gamma\) be a group generated by a finite symmetric set \(S\). Let
\(\mu\in\Prob(\Gamma)\) be a generating measure, and let
\(\nu\in\Prob(\Sub_\Gamma)\) be a \(\mu\)-stationary random subgroup. Then
\[
|\Ends(\Sch(H\backslash\Gamma,S))|\in\{0,1,2,\infty\}
\]
for \(\nu\)-almost every \(H\).
\end{theorem}

\begin{remark}
We do not prove that the case of infinity is uncountable.
\end{remark}
\begin{proof}
    Let
    \[
    E=\{H\in\Sub_\Gamma:2<|\Ends(H)|<\infty\}.
    \]
    Assume toward contradiction that \(\nu(E)>0\).
    For \(k\ge 3\), set
    \[
    E_k=\{H\in\Sub_\Gamma:|\Ends(H)|=k\}.
    \]
    Since
    \[
    E=\bigcup_{k=3}^\infty E_k,
    \]
    there exists \(k_0\ge 3\) such that \(\nu(E_{k_0})>0\).

For \(n\ge 1\), define \(C_{k_0,n}\) as in Lemma~\ref{set_is_a_core}. We have
    $$E_{k_0} = \bigcup_{n=1}^\infty{C_{k_0,n}}$$
    Indeed, if $\Sch(H)$ has $k_0$ ends, there is $N_H \ge 1$ such that for every $n \ge N_H$, $\Sch(H)\setminus B_n(H)$ has exactly $k_0$ infinite connected components.
    Therefore there is $n_0\ge 1$ such that $\nu(C_{k_0,n_0})>0$. 
    By Lemma~\ref{set_is_a_core} $C_{k_0,n_0}$ is a core.
    
    So $C_{k_0,n_0}$ is a core with $\nu(C_{k_0,n_0})>0$ but $\nu$ is stationary, then by Theorem~\ref{no-core} $\nu$-a.e. $H \in C_{k_0,n_0}$, has finite conjugacy class $|H^{\G}|  < \infty$.
    But every $H \in C_{k_0,n_0}$ has 
    $2<|\Ends(H)|<\infty$,  contradicting Lemma~\ref{finite_ends_infinite_orbit}.
\end{proof}

\section{The Boomerang counter example}

A Boomerang subgroup, as introduced in~\cite{Boomerang} is a subgroup $\Delta < \Gamma$ such that for every $g\in \G$ the sequence $(g^n\Delta g^{-n})_{n=0}^\infty$ has a subsequence which converges to $\Delta$. If $\lambda \in \Prob(\subg)$ is an IRS, then $\lambda$-a.e. subgroup is Boomerang.

That is, Boomerang is a topological property (in $\subg$), it is ``the area in which IRSs are happening''. Hence, classifying the Boomerang subgroups, yields a classification of the possible IRSs, bypassing the very nature of an IRS as a probabilistic object.
It is therefore natural to ask whether IRS end-rigidity  (the IRS version of Theorem~\ref{ends_theorem}) can be proved by first proving a pointwise statement for all Boomerang subgroups. The following example shows that this strategy cannot work.

\begin{theorem}
    For any integer $k \ge 3$ or $k = \aleph_0$, there exists a Boomerang subgroup of the free group $\mathbb{F}_3$ whose Schreier graph has exactly $k$ ends.
\end{theorem}

\begin{proof}
    Let \(p\) be sufficiently large and odd, and let \(\Gamma\) be an infinite 3-generated Burnside group of exponent \(p\). Let
\[
\varphi:\mathbb{F}_3=\langle a_1,a_2,a_3\rangle\longrightarrow \Gamma
\]
be an epimorphism, and write \(x_i=\varphi(a_i)\). Let \(\mathcal G=\operatorname{Cay}(\Gamma,\{x_1^{\pm 1},x_2^{\pm 1},x_3^{\pm 1}\})\).

    Note that $\mathcal{G}$ is 1-ended. Indeed, by Stallings' theorem on the ends of groups, a finitely generated group with more than one end must either be virtually $\mathbb{Z}$ (if it has 2 ends) or split over a finite subgroup (if it has infinitely many ends). Both cases strictly require the group to contain elements of infinite order. Since $\Gamma$ is an infinite torsion group, neither cases is possible, and so, $\mathcal G$ has exactly 1 end.

    Let $Z = \mathbb{Z}/k\mathbb{Z}$ if $k < \infty$, and $Z = \mathbb{Z}$ if $k = \aleph_0$. We define a new graph $\mathcal{G'}$ consisting of $k$ disjoint copies of $\mathcal{G}$, indexed by $Z$, with the following modifications: maintain all vertices from the original copies, and keep all edges except for the $x_1$-edges originating at the identity in each copy: for each $i \in Z$, we remove the $x_1$-edge from $e^{(i)}$ to $x_1^{(i)}$ and instead add an $x_1$-edge connecting $e^{(i)}$ to $x_1^{(i+1)}$. This "cut-and-paste" operation preserves the regularity of the graph as a labeled Schreier graph.

\begin{figure}[htpb]
    \centering
    \begin{tikzpicture}[scale=0.9, thick]
        \coordinate (C1) at (0,0);
        \coordinate (C2) at (4,0);
        \coordinate (C3) at (8,0);

        \draw (C1) circle (1.5cm);
        \draw (C2) circle (1.5cm);
        \draw (C3) circle (1.5cm);

        \node at (0, 1.8) {Copy $i-1$};
        \node at (4, 1.8) {Copy $i$};
        \node at (8, 1.8) {Copy $i+1$};

        \filldraw (0, -0.5) circle (1.5pt) node[below] {$e^{(i-1)}$};
        \filldraw (4, -0.5) circle (1.5pt) node[below] {$e^{(i)}$};
        \filldraw (8, -0.5) circle (1.5pt) node[below] {$e^{(i+1)}$};

        \filldraw (0, 0.5) circle (1.5pt) node[above] {$x_1^{(i-1)}$};
        \filldraw (4, 0.5) circle (1.5pt) node[above] {$x_1^{(i)}$};
        \filldraw (8, 0.5) circle (1.5pt) node[above] {$x_1^{(i+1)}$};

        \draw[->, >=stealth, dashed, gray] (0, -0.5) -- (0, 0.4);
        \draw[->, >=stealth, dashed, gray] (4, -0.5) -- (4, 0.4);
        \draw[->, >=stealth, dashed, gray] (8, -0.5) -- (8, 0.4);

        \draw[->, >=stealth, red, thick] (-2, -0.5) -- (-0.1, 0.45); 
        
        \draw[->, >=stealth, red, thick] (0, -0.5) -- (3.9, 0.45);
        \draw[->, >=stealth, red, thick] (4, -0.5) -- (7.9, 0.45);
        
        \draw[->, >=stealth, red, thick] (8, -0.5) -- (10, 0.45);

        \node at (-2.5, 0) {\dots};
        \node at (10.5, 0) {\dots};

    \end{tikzpicture}
    \caption{The rewiring construction of $\mathcal{G}'$. The original $x_1$-edges originating at the identity vertices (dashed gray) are removed. Instead, new $x_1$-edges (solid red) are added to connect $e^{(j)}$ to $x_1^{(j+1)}$ in the adjacent copy, linking the $k$ components.}
    \label{fig:boomerang_rewiring}
\end{figure}
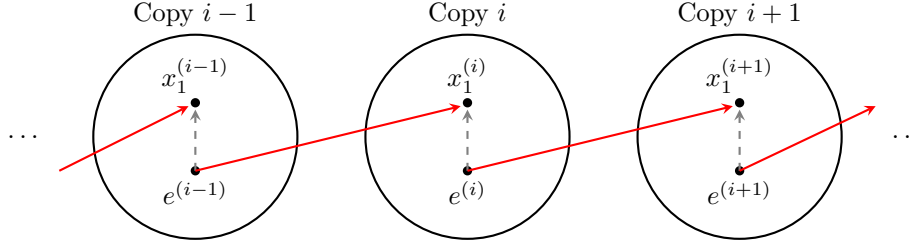

    Let $\Delta \le \mathbb{F}_3$ be the subgroup corresponding to the rooted Schreier graph $(\mathcal{G'}, e^{(0)})$. Because we connected $k$ copies of a 1-ended graph in a structure of a cycle-like for $k<\infty$ or a bi-infinite chain-like for $k=\aleph_0$, in either case, the resulting graph $\mathcal{G'}$ has exactly $k$ ends. It only remains to show that $\Delta$ is a Boomerang subgroup.

    Take any $g \in \mathbb{F}_3$. Since $\Gamma$ is a torsion group of exponent $p$, $\varphi(g)^p = e$, which implies that $g^p \in \ker(\varphi)$. Consequently, $g^{np} \in \ker(\varphi)$ for every integer $n$. 

    When we consider $\mathcal{G'}$ as a quotient of the Cayley graph of $\mathbb{F}_3$,
    we also consider $\mathcal{G^\prime}$ as a cover of the graph $\mathcal{G}$ where $\pi$ is the cover map.
    
    We recall that conjugating the subgroup corresponds to re-rooting the root of the Schreier graph.

    Re-rooting the Schreier graph at the endpoint reached by the word \(g^{-pn}\) corresponds to the subgroup
\[
g^{pn}\Delta g^{-pn}.
\]
    
    Consider the path in \(\mathcal G'\) starting at \(e^0\) and labeled by \(g^{-pn}\). Since \(\pi\) is label-preserving and \(\varphi(g^{-pn})=e\), the endpoint of the path in $\mathcal{G}$ lies over the vertex \(e\in\Gamma\), so the path in \(\mathcal G'\) ends in \(e^{(i_n)}\) for some \(i_n\in Z\).

    Since as rooted labeled Schreier graphs, $(\mathcal{G^\prime}, e^{(0)})$ and $(\mathcal{G^\prime}, e^{(i_n)})$ are isomorphic, \emph{the corresponding subgroups
are equal}.
Therefore
\[
g^{pn}\Delta g^{-pn}=\Delta
\]
for every \(n\ge 1\).

    Therefore, the sequence of conjugate subgroups has a subsequence $(g^{np}\Delta g^{-np})_{n=1}^\infty$ which is equals to $\Delta$ and in particular converges to $\Delta$. This establishes that $\Delta$ is a Boomerang subgroup.
\end{proof}

\printbibliography

\end{document}